\documentclass[11pt,a4paper]{article}
\usepackage[utf8]{inputenc}
\usepackage[T1]{fontenc}
\usepackage{geometry}
\usepackage{amsmath, amssymb, amsthm}
\usepackage{graphicx}
\usepackage{hyperref}
\usepackage{xcolor}
\usepackage{mathrsfs}
\usepackage{enumitem}

\usepackage{tikz}
\usetikzlibrary{arrows.meta,decorations.markings,decorations.pathreplacing,positioning,fit,backgrounds,matrix,calc}
\usepackage{blkarray}

\usepackage{mathtools}
\usepackage{comment}

\usepackage{fontspec}
\usepackage[frozencache]{minted}
\usepackage[most]{tcolorbox}
\tcbuselibrary{minted,skins,breakable}
\definecolor{juliagreen}{HTML}{389826}
\setminted{
  fontsize=\footnotesize,
  autogobble,
  breaklines=true,
  breakindent=8pt,
  breaksymbolleft={\tiny\textcolor{black!55}{$\hookrightarrow$}},
  breaksymbolright={},
  style=friendly
}

\newtcblisting{juliacode}{
  listing only,
  listing engine=minted,
  minted language=julia,
  enhanced jigsaw,
  colback=white,
  colframe=black,
  boxrule=0pt,
  toprule=0.4pt, bottomrule=0.4pt, leftrule=0pt, rightrule=0pt,
  arc=0pt,
  left=1mm, right=1mm, top=0.6mm, bottom=0.6mm,
  breakable
}

\newtcblisting{juliarepl}{
  listing only,
  listing engine=minted,
  minted language=jlcon,
  enhanced jigsaw,
  colback=white,
  colframe=juliagreen,
  boxrule=0pt,
  toprule=0pt, bottomrule=0pt, rightrule=0pt,
  leftrule=2pt,
  arc=0pt,
  left=2mm, right=1mm, top=0.6mm, bottom=0.6mm,
  breakable
}

\newcommand{\codepair}{\vspace*{-0.5\baselineskip}}

\newtcolorbox{juliasplit}{
  sidebyside,
  lefthand width=0.5\linewidth,
  sidebyside gap=3mm,
  sidebyside align=top seam,
  enhanced jigsaw,
  colback=white,
  colframe=black,
  boxrule=0pt,
  toprule=0.4pt, bottomrule=0.4pt, leftrule=0pt, rightrule=0pt,
  arc=0pt,
  left=1mm, right=1mm, top=0.6mm, bottom=0.6mm,
  segmentation style={solid, line width=0.4pt, black}
}

\newtheorem{theorem}{Theorem}[section]

\newtheorem{remark}[theorem]{Remark}

\newtcolorbox{practicalbox}{
  enhanced jigsaw,
  colback=white,
  colframe=black,
  boxrule=0pt,
  leftrule=1pt,
  toprule=0pt, bottomrule=0pt, rightrule=0pt,
  arc=0pt,
  left=10pt, right=0pt, top=4pt, bottom=4pt,
  breakable
}
\theoremstyle{definition}
\newtheorem{ppropinner}[theorem]{Proposition}
\newtheorem{pleminner}[theorem]{Lemma}
\theoremstyle{plain}
\NewDocumentEnvironment{pproposition}{o}
  {\begin{practicalbox}\IfValueTF{#1}{\begin{ppropinner}[#1]}{\begin{ppropinner}}}
  {\end{ppropinner}\end{practicalbox}}
\NewDocumentEnvironment{plemma}{o}
  {\begin{practicalbox}\IfValueTF{#1}{\begin{pleminner}[#1]}{\begin{pleminner}}}
  {\end{pleminner}\end{practicalbox}}

\newcommand{\R}{\mathbb{R}}
\newcommand{\C}{\mathbb{C}}
\newcommand{\Z}{\mathbb{Z}}
\newcommand{\N}{\mathbb{N}}

\newcommand{\bx}{\bar{x}}

\newcommand{\tx}{x^\star}

\newcommand{\norm}[1]{\left\|#1\right\|}

\newcommand{\bydef}{\stackrel{\textnormal{\tiny def}}{=}}

\title{Computer-Assisted Proofs in Dynamical Systems:\\
A Case Study of a Heteroclinic Orbit in the Shimizu--Morioka System}
\author{
Olivier H\'{e}not
\thanks
{National Taiwan University, Department of Mathematics, No. 1 Sec. 4 Roosevelt Rd., 10617 Taipei, Taiwan. {\tt olivierhenot@ntu.edu.tw}.}
\and
Akitoshi Takayasu
\thanks
{University of Tsukuba, Institute of Systems and Information Engineering, 1-1-1 Tennodai, Tsukuba, Ibaraki 305-8573, Japan. {\tt takitoshi@risk.tsukuba.ac.jp}.}
}
\date{}

\begin{document}

\maketitle

\begin{abstract}
The radii polynomial approach is an a posteriori validation method based on the contraction of a quasi-Newton operator.
We apply this strategy to give a computer-assisted proof of a transverse heteroclinic orbit in the Shimizu--Morioka system, validating the equilibria and eigenpairs, the local invariant manifolds via the parameterization method, and the connecting orbit via a boundary-value problem.
For each subproblem we present a four-step procedure: $(i)$ zero-finding formulation, $(ii)$ approximate zero, $(iii)$ approximate inverse, and $(iv)$ bound estimates.
More broadly, this highlights the unifying structure behind the a posteriori validation method.
Alongside the analysis, we include code snippets implemented in Julia \cite{Julia-2017} using the \texttt{RadiiPolynomial} \cite{RadiiPolynomial.jl} library.
\end{abstract}



\section{Introduction}

The Shimizu--Morioka system~\cite{ShimizuMorioka}
\begin{equation}\label{eq:shimizu_morioka}
\dot{x} = y, \qquad \dot{y} = x - a y - x z, \qquad
\dot{z} = -b z + x^2,
\end{equation}
where $a, b$ are real parameters, illustrates how complex dynamics can arise even in simple models.
As with the Lorenz system \cite{Lorenz}, system~\eqref{eq:shimizu_morioka} exhibits a \emph{butterfly} shaped strange attractor, where trajectories visit the two \emph{wings} in any prescribed symbolic sequence.
The invariant manifolds, attached to equilibria or periodic orbits, are central to the global picture of the dynamics.
These objects, however, are notoriously difficult to obtain analytically, and in general come with limited (if any) quantitative information.
This makes the study of invariant manifolds and their intersection a particularly compelling instance of a fundamental problem in dynamical systems that benefits from the assistance of the computer.

\begin{figure}
    \centering
    \includegraphics[width=\textwidth]{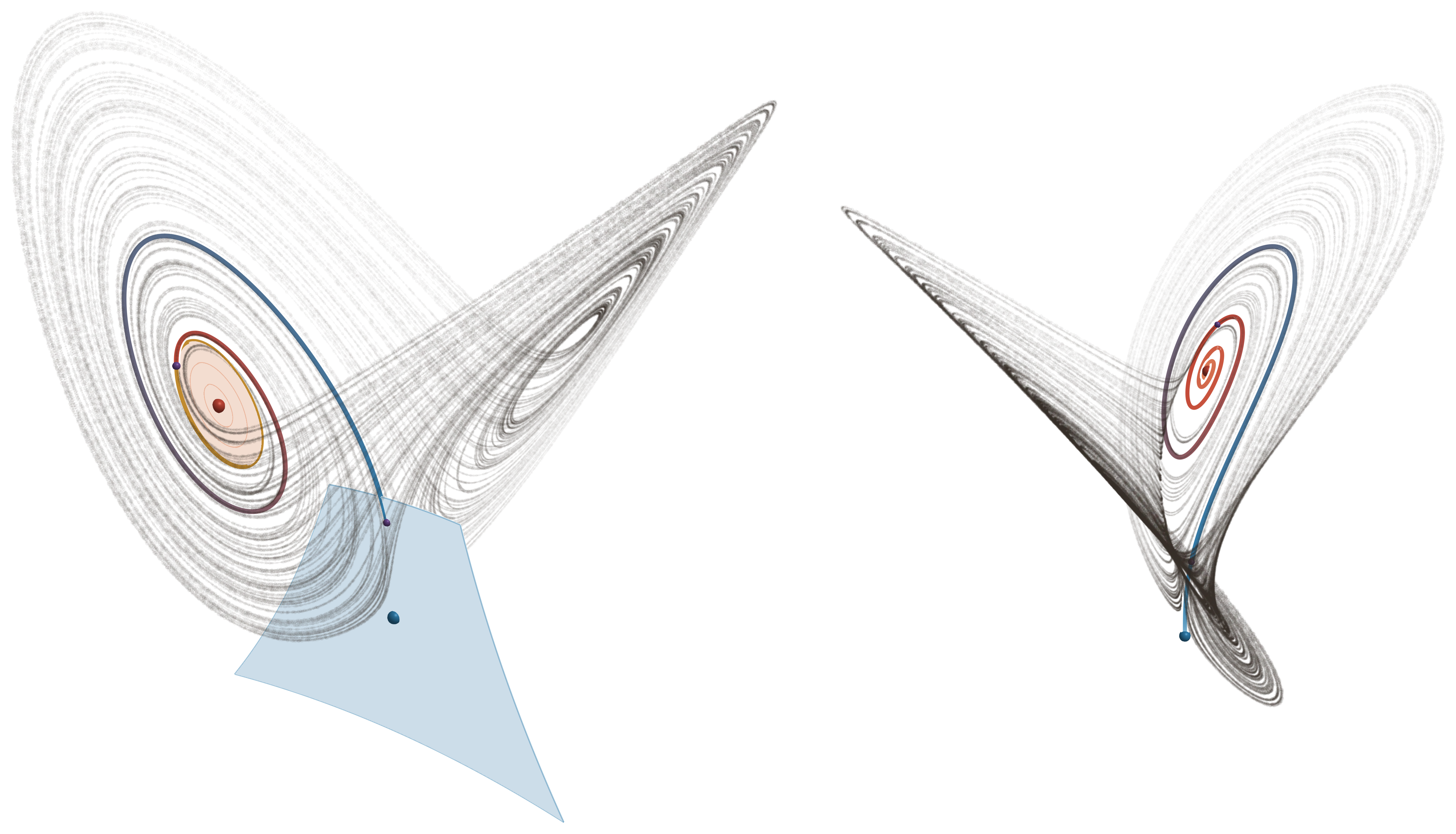}
    \caption{Transverse heteroclinic orbit connecting two saddle equilibria of the Shimizu--Morioka system, inside the strange attractor (shown in light grey). Left: the orbit with the local unstable manifold of $c_1^\star$ (light red) and the local stable manifold of $c_0^\star$ (light blue). Right: another viewpoint.}
    \label{fig:het}
\end{figure}

In this paper, we use the existence proof of a transverse heteroclinic orbit connecting two saddle equilibria of the Shimizu--Morioka system (see Figure~\ref{fig:het}) to guide the reader through the key ideas and techniques of a posteriori validation based on the radii polynomial approach.
While the underlying theory is well established (see, e.g., \cite{LessardMirelesJamesReinhardt2014}), our exposition pursues a more structural aim.
We highlight that each subproblem of the proof (the validation of equilibria and eigenpairs, of local invariant manifolds, and of the connecting orbit) fits into the same four-step procedure, thereby revealing the common structure of the computer-assisted proof framework in both finite- and infinite-dimensional settings, and isolating what is problem specific from what belongs to the framework.
In parallel, we provide code snippets accompanying the mathematical analysis to illustrate explicitly how it translates into concrete practical implementations.
As much as possible, we keep the discussion elementary so that the article remains approachable for newcomers; more experienced readers will find opportunities to improve the code (performance, memory management) and to sharpen the analytical bounds.

The difficulty of describing solutions to nonlinear differential equations has long motivated the development of practical theorems and algorithms to obtain approximate solutions.
The fact that these efforts predate the digital computer by centuries \cite{Goldstine} highlights the enduring need for reliable computational methods.
The influence of the Japanese school of applied mathematics on the development of computer-assisted proofs must be acknowledged.
Having embraced these methods early on, the Japanese school pioneered the field in many ways \cite{Sunaga1958, Urabe1965, Nakao1988, Yamamoto1998}, and continues to play a leading role today \cite{NakaoYamamoto1998, Oishi1995, OishiRump2002, Oishi2, Rump1999, Rump2010, Nakao-book, Watanabe}.
The research and applications of computer-assisted proofs are rapidly growing \cite{Tucker2002, RadiiPolynomial, WrightConjecture, JonesConjecture, vdBergLessard2018, GomezSerrano2019, SekineNakaoOishi2020, TakayasuLessard2022, CalGarHenLesMir24, TakayasuPDE}.

The article is organized as follows.
Section~\ref{sec:setup} states the a posteriori validation framework used throughout the paper.
Section~\ref{sec:eq_eig} validates equilibria of the Shimizu--Morioka system and their stable and unstable eigenpairs.
Section~\ref{sec:manifold} validates Taylor series expansions of local parameterizations of the invariant manifolds via the parameterization method.
Section~\ref{sec:heteroclinic} validates a Chebyshev series expansion of a segment of the transverse heteroclinic orbit connecting the two equilibria.
The final section concludes with comments on the scope of the approach and possible extensions.

\paragraph{Software resources.}
The code accompanying this article, available at~\cite{Code}, is written in Julia~\cite{Julia-2017}, using the \texttt{RadiiPolynomial}~\cite{RadiiPolynomial.jl} and \texttt{IntervalArithmetic}~\cite{IntervalArithmetic.jl} software libraries.
We also made use of the suite of ODE solvers from \texttt{DifferentialEquations} \cite{DifferentialEquations.jl}.
The figure in this article was made with \texttt{Makie} \cite{Makie.jl}.


\subsection{The radii polynomial approach}\label{sec:setup}

Many problems in dynamical systems are naturally formulated in the form of $F(x) = 0$.
We begin this article by presenting a strategy to prove the existence of a regular zero of $F$.

\noindent
Let
\begin{enumerate}[label=(\roman*)]
\item $\bx$ be a numerical approximation of the zero (i.e., $F(\bx) \approx 0$ numerically), and
\item $A$ a numerical approximation of the inverse of $DF(\bx)$ (i.e., $A \approx DF(\bx)^{-1}$ numerically).
\end{enumerate}
It is important that these can be obtained accurately by numerical computation; in this sense, computer-assisted proofs rest on the foundations of numerical analysis.
The following theorem gives sufficient conditions to prove the contraction of the quasi-Newton operator $x \mapsto x - A F(x)$ in an \emph{explicit closed ball} of $\bx$.
We denote by $\mathscr{B}(X)$ the space of bounded linear operators from $X$ to itself, while $B(\bx, r)$ denotes the closed ball in $X$ centered at $\bx$ with radius $r$.
By convention, if $r = \infty$, then $B(\bx, r) = X$.

\begin{theorem}[Radii Polynomial Theorem]\label{thm:rpa}
Let $X$ be a Banach space, $\bx \in X$, $F : X \to X$ a $C^1$ map, and $A \in \mathscr{B}(X)$ an injective map.
Fix $R \in [0, \infty]$, and let $Y, Z = Z(R) \ge 0$ be two positive constants satisfying
%
\begin{equation}
\norm{A F(\bx)}_X \le Y, \qquad
\sup_{x \in B(\bx, R)} \norm{I - A DF(x)}_{\mathscr{B}(X)} \le Z.
\end{equation}
If $Z < 1$,
%
%
then, for any $r \ge 0$ such that
\begin{equation}
\frac{Y}{1 - Z} \le r \le R,
\end{equation}
the fixed-point operator $x \mapsto x - A F(x)$ is a contraction on $B(\bx, r)$, so that there exists a unique zero $\tx \in B(\bx, r)$ of $F$.
\end{theorem}

\begin{proof}
For $x \in B(\bx, r)$, we have by the mean value inequality
\begin{align*}
\norm{x - A F(x) - \bx}_X
&\le \norm{A F(\bx)}_X + r \int_0^1 \norm{I - A DF(\bx + s(x - \bx))}_{\mathscr{B}(X)} \, \mathrm{d}s \\
&\le Y + r Z \\
&\le r,
\end{align*}
which proves that $x \mapsto x - A F(x)$ is a map from $B(\bx, r)$ to itself.
Moreover, for $x,y \in B(\bx, r)$,
\begin{align*}
\norm{x - A F(x) - (y - A F(y))}_X
&\le \left(\sup_{\xi \in B(\bx, r)} \norm{I - A DF(\xi)}_{\mathscr{B}(X)} \right) \norm{x - y}_X \\
&\le Z \norm{x - y}_X,
\end{align*}
which proves that $x \mapsto x - A F(x)$ is a contraction on $B(\bx, r)$ since $Z < 1$.
\end{proof}

\noindent
We interpret
\begin{enumerate}
\item $Y$ as measuring the quality of the approximate zero $\bx$ of $F$.
\item $Z = Z(R)$ as measuring the quality of the approximate inverse $A$ at $\bx$, as well as quantifying the variation of the fixed-point operator $x \mapsto x - AF(x)$ over $B(\bx, R)$.
\item $R$ as an a priori error threshold, determining the ball for the $Z(R)$ bound. In special cases, one can take $R = \infty$, for instance when $F$ is affine (so that $DF$ is constant).
\end{enumerate}

For a finite-dimensional space $X$, the Radii Polynomial Theorem~\ref{thm:rpa} can be applied directly using interval arithmetic, as detailed in Section~\ref{sec:equilibria}.
Importantly, the injectivity of $A$ need not be established a priori.
Indeed, the condition $Z < 1$ implies that $A$ is surjective, and, since $X$ is finite-dimensional, $A$ is a square matrix that must also be injective.

In contrast, when $X$ is an infinite-dimensional function space, the procedure is less evident and requires analytical estimates to derive computable formulas for $Y$ and $Z$.
In this paper, we address the following questions: \emph{How can analytic functions be modeled in a sequence space? How can nonlinearities be controlled in this setting? How can the estimates for the bounds $Y$ and $Z$ be reduced to a finite set of computations that can be carried out by the computer?}


\section{Validated computation of equilibria and eigenspaces}\label{sec:eq_eig}

The vector field of system~\eqref{eq:shimizu_morioka} is given by
\begin{equation}\label{eq:vf}
f(x, y, z) \bydef
\begin{pmatrix}
y \\
x-ay-xz \\
-b z + x^2
\end{pmatrix}, \qquad \text{and we choose the parameter values } a = \frac{3}{4}, \, b = \frac{9}{20}.
\end{equation}
The implementation of the vector field $f$ and its Jacobian $Df$ is as follows.
\begin{juliasplit}
\begin{minted}{julia}
function f(u, params)
    a, b = params
    x, y, z = u
    return [y
            x - a*y - x*z
            -b*z + x^exact(2)]
end
\end{minted}
\tcblower
\begin{minted}{julia}
function Df(u, params)
    a, b = params
    x, y, z = u
    return [zero(x)     one(y)     zero(z)
            exact(1)-z  -a*one(y)  -x
            exact(2)*x  zero(y)    -b*one(z)]
end
\end{minted}
\end{juliasplit}

\begin{remark}\label{rem:generic_f_Df}
The functions \texttt{f} and \texttt{Df} above are written in a generic form so that they can be reused throughout the paper.
In particular, two implementation choices deserve comment:
\begin{enumerate}
\item
Typed-in literal numbers are wrapped with \texttt{exact}.
The \texttt{exact} function, provided by \texttt{Interval\allowbreak Arithmetic} \cite{IntervalArithmetic.jl}, declares that the underlying number is to be treated as mathematically exact.
This means that, without \texttt{exact}, mixing intervals with, e.g., floating-point numbers produces a result flagged as \emph{NG (Not Guaranteed)}.
Thus this let \texttt{f} and \texttt{Df} work properly for both floating-point and interval inputs.
\item
The functions \texttt{zero} and \texttt{one} return, respectively, the additive and multiplicative identities in the type of their argument.
This ensures that \texttt{f} and \texttt{Df} are compatible with the sequences introduced in Sections~\ref{sec:manifold} and~\ref{sec:heteroclinic}.
\end{enumerate}
\end{remark}

\noindent
We make use of the \texttt{RadiiPolynomial} and \texttt{LinearAlgebra} libraries:
\begin{juliarepl}
julia> using RadiiPolynomial, LinearAlgebra
\end{juliarepl}
\noindent
Note that the \texttt{RadiiPolynomial} library automatically loads the \texttt{IntervalArithmetic} library \cite{IntervalArithmetic.jl}, making interval arithmetic immediately available.
We can now enclose rigorously the parameters $a, b$ using interval arithmetic:
\begin{juliarepl}
julia> parameters = (interval(3)/interval(4), interval(9)/interval(20)) # (a, b)
([0.75, 0.75], [0.45, 0.45])
\end{juliarepl}
%


\subsection{Computation of the equilibria}
\label{sec:equilibria}

The first task consists in identifying two distinct equilibria $c_0^\star, c_1^\star \in \R^3$ such that $f(c_0^\star) = f(c_1^\star) = 0$.
By inspection, $c_0^\star = (0, 0, 0)$ is an equilibrium of~\eqref{eq:vf}.
For the non-trivial equilibrium $c_1^\star$, the vector field~\eqref{eq:vf} is simple enough that an explicit formula could be obtained by hand, namely,
\begin{equation*}
c_1^\star = (\sqrt{b}, 0, 1)
\end{equation*}
and we note that it is paired with the equilibrium $(-\sqrt{b}, 0, 1)$ not studied in this article.
However, we pretend not to know this formula for $c_1^\star$ and use its validation as an easy first application of the Radii Polynomial Theorem~\ref{thm:rpa}.

The \emph{equilibrium problem} involves the vector field $f$, its Jacobian $Df$ and the parameters $a, b$.
We begin by creating a data structure \texttt{EquilibriumProblem} bundling together these elements.
\begin{juliacode}
struct EquilibriumProblem
    f
    Df
    parameters
end
\end{juliacode}
\codepair
\begin{juliarepl}
julia> pb = EquilibriumProblem(f, Df, parameters)
EquilibriumProblem(f, Df, ([0.75, 0.75], [0.45, 0.45]))
\end{juliarepl}
\noindent
The next four paragraphs provide the steps to apply and implement the Radii Polynomial Theorem~\ref{thm:rpa}.

\paragraph{Step 1: Defining the zero-finding problem.}
An equilibrium is a zero of the vector field $f : \R^3 \to \R^3$, where we choose to endow $\R^3$ with the $1$-norm
\begin{equation}
\norm{c}_{\R^3} \bydef \sum_{i=1}^3 |c^{(i)}|, \qquad \text{for all } c = (c^{(1)}, c^{(2)}, c^{(3)})\in \R^3.
\end{equation}
Hence, we apply the Radii Polynomial Theorem~\ref{thm:rpa} with $F = f$ on $X = \R^3$.
The implementation is therefore straightforward.
\begin{juliacode}
F(pb::EquilibriumProblem, c) = pb.f(c, pb.parameters)

DF(pb::EquilibriumProblem, c) = pb.Df(c, pb.parameters)
\end{juliacode}

\paragraph{Step 2: Computing the approximate zero (with floating-point arithmetic).}
We compute an accurate approximation $\bar{c} \in \R^3$ of the equilibrium using Newton's method, which is executed via the \texttt{newton} function available from the software library \texttt{RadiiPolynomial}.
Importantly, interval arithmetic is not needed at this stage; we replace the interval parameters in \texttt{EquilibriumProblem} by their midpoints.
\begin{juliacode}
function approximate_zero(pb_approx::EquilibriumProblem, c_init)
    F_DF(c) = F(pb_approx, c), DF(pb_approx, c)
    c_bar, newton_success = newton(F_DF, c_init)
    return c_bar, newton_success
end
\end{juliacode}
\codepair
\begin{juliarepl}
julia> pb_approx = EquilibriumProblem(pb.f, pb.Df, mid.(pb.parameters))
EquilibriumProblem(f, Df, (0.75, 0.44999999999999996))

julia> c1_bar, newton_success = approximate_zero(pb_approx, [1.0, 0.0, 1.0])
([0.6708203932499369, 0.0, 1.0], true)
\end{juliarepl}

\paragraph{Step 3: Constructing the approximate inverse (with floating-point arithmetic).}
The construction of $A$ is straightforward in finite dimensions.
Indeed, $DF(\bar{c}) = Df(\bar{c})$ is a matrix and we can rely on a numerical algorithm to produce an approximate inverse matrix.
We therefore compute $A \approx DF(\bar{c})^{-1}$ numerically using Julia's built-in \texttt{inv} function.
\begin{juliacode}
function approximate_inverse(pb_approx::EquilibriumProblem, c_bar)
    return inv(DF(pb_approx, c_bar))
end
\end{juliacode}
\codepair
\begin{juliarepl}
julia> A = approximate_inverse(pb_approx, c1_bar)
3×3 Matrix{Float64}:
 -0.375    -0.5      0.745356
  1.0       0.0      0.0
 -1.11803  -1.49071  0.0
\end{juliarepl}

\paragraph{Step 4: Estimating the bounds (with interval arithmetic).}
The $Y$ bound
\begin{equation}
Y = \norm{A F ( \bar{c} )}_{\R^3}
\end{equation}
consists of taking the $1$-norm of a matrix-vector product, which is accomplished directly using interval arithmetic.

For a fixed $R \ge 0$, the calculation of $Z = Z(R)$ requires taking a supremum over the closed ball $B(\bar{c}, R)$.
As $B(\bar{c}, R)$ sits in the finite-dimensional space $X = \R^3$ endowed with the $1$-norm topology, we can enclose it by a three-dimensional \emph{interval box}:
\begin{equation}
B(\bar{c}, R)
\subset \mathrm{Box}(\bar{c}, R)
= \begin{pmatrix}[\bar{c}^{(1)} - R, \bar{c}^{(1)} + R] \\ [\bar{c}^{(2)} - R, \bar{c}^{(2)} + R] \\ [\bar{c}^{(3)} - R, \bar{c}^{(3)} + R] \end{pmatrix}.
\end{equation}
We make the heuristic choice that $R \sim 10 Y$, and compute
\begin{equation}
Z = \norm{I - A Df( \mathrm{Box}(\bar{c}, R) )}_{\mathscr{B}(\R^3)}.
\end{equation}

Once $Y$ and $Z$ have been computed, we can check the contraction criterion $Z < 1$ and the condition $Y/(1-Z) \le R$.
Upon their verification, the computer-assisted proof has succeeded, and the Radii Polynomial Theorem~\ref{thm:rpa} yields an \emph{interval of existence} whose infimum $r = Y/(1-Z)$, is a rigorous a posteriori error bound, with respect to the norm $\norm{\,\cdot\,}_{\R^3}$, on the approximate equilibrium $\bar{c}$.
This verification is done using the \texttt{interval\_of\_existence} function provided by the \texttt{RadiiPolynomial} library, which returns the interval of existence together with a Boolean value \texttt{true}, or \texttt{false}, signaling the success, or failure, of the computer-assisted proof.
\begin{juliacode}
Y_bound(pb::EquilibriumProblem, c_bar, A) = norm(A * F(pb, c_bar), 1)

function Z_bound(pb::EquilibriumProblem, c_bar, A, R)
    box_c = interval.(c_bar, R; format = :midpoint)
    return opnorm(interval(I) - A * DF(pb, box_c), 1)
end
\end{juliacode}
\codepair
\begin{juliarepl}
julia> Y = Y_bound(pb, interval(c1_bar), interval(A))
[0.0, 8.27511e-17]

julia> R = interval(10sup(Y))
[8.27511e-16, 8.27511e-16]

julia> Z = Z_bound(pb, interval(c1_bar), interval(A), R)
[6.93889e-17, 3.32242e-15]

julia> ie, proof_success = interval_of_existence(Y, Z, R)
([8.27511e-17, 8.27511e-16]_com, true)

julia> r_inf = inf(ie)
8.275113844157605e-17

julia> c1_star = interval.(c1_bar, r_inf; format = :midpoint)
3-element Vector{Interval{Float64}}:
  [0.67082, 0.67082]
 [-8.27511e-17, 8.27511e-17]
  [1.0, 1.0]
\end{juliarepl}
In the above last line of code evaluation, we represent the mathematically exact equilibrium $c^\star$, proven to lie in $B(\bar{c}, r)$, as the interval box
\begin{equation}
c^\star
\in \mathrm{Box}(\bar{c}, r)
= \begin{pmatrix}[\bar{c}^{(1)} - r, \bar{c}^{(1)} + r] \\ [\bar{c}^{(2)} - r, \bar{c}^{(2)} + r] \\ [\bar{c}^{(3)} - r, \bar{c}^{(3)} + r] \end{pmatrix}.
\end{equation}
%


\subsection{Computation of the eigenspaces}
\label{sec:eigenspaces}

Having determined the two equilibria $c_0^\star, c_1^\star \in \R^3$, we now recover the stable and unstable eigenspaces.
To achieve this, for each $i=0,1$, we solve the eigenvalue problem
\begin{equation}
Df(c_i^\star) v = \lambda v.
\end{equation}
Provided that the eigenvalues are simple, an adequate zero-finding problem is
\begin{equation}
F_\mathrm{eig}(v, \lambda) \bydef
\begin{pmatrix}
Df( c_i^\star )v - \lambda v \\
v^{(l_*)} - 1
\end{pmatrix} = 0, \qquad (v, \lambda) \in \C^3 \times \C,
\end{equation}
for some fixed index $l_* \in \{1, 2, 3\}$.
We endow $X = \C^3 \times \C$ with the $1$-norm $\norm{(v, \lambda)}_X \bydef \sum_{i=1}^3 |v^{(i)}| + |\lambda|$.
The equation $v^{(l_*)} - 1 = 0$ is a \emph{normalization condition} to isolate the eigenvector in $\ker(Df( c_i^\star ) - \lambda I)$, which is enough when the kernel is one-dimensional and $l_*$ is chosen such that the $l_*$-th component of the eigenvector is nonzero.

This is again a finite-dimensional problem, and we can apply the Radii Polynomial Theorem~\ref{thm:rpa} with $F = F_\mathrm{eig}$ on $X = \C^3 \times \C$, following the same steps as we did in the proof of the equilibria.
Since the procedure is virtually identical, we do not repeat it, and only summarize the obtained result:
\begin{enumerate}
\item $c_0^\star$ is a saddle equilibrium with \emph{two real stable eigenvalues} and one real unstable eigenvalue.
\item $c_1^\star$ is a saddle equilibrium with one stable eigenvalue and \emph{two complex conjugate unstable eigenvalues}.
\end{enumerate}
For the complex conjugate pair, we validate one eigenpair and take the other to be its complex conjugate.


\section{Validated computation of invariant manifolds}\label{sec:manifold}

We now turn to the nonlinear analogues of the eigenspaces, which will be used to set up the boundary-value problem proving the existence of a connecting orbit.
The classical (Un)Stable Manifold Theorem guarantees the existence of a local manifold, invariant under the flow, as the image of a graph over the (un)stable eigenspace \cite{Chicone}.
To state things briefly, denoting by $\varphi_t : \R^3 \to \R^3$ the flow associated with the Shimizu--Morioka system~\eqref{eq:shimizu_morioka}, the local stable and unstable manifolds consist of the set of initial conditions in phase space such that
\begin{subequations}
\begin{align}
W^{s}_{\mathrm{loc}}(c) &= \left\{ \phi \in \text{neighbourhood of $c$} \, : \, \lim_{t\to +\infty} \varphi_t(\phi) = c \right\}, \\
W^{u}_{\mathrm{loc}}(c) &= \left\{ \phi \in \text{neighbourhood of $c$} \, : \, \lim_{t\to -\infty} \varphi_t(\phi) = c \right\}.
\end{align}
\end{subequations}
Under some \emph{mild non-resonance condition} (detailed below), it is in fact possible to parameterize the local manifold without requiring it to be a graph; this strategy is called the \emph{parameterization method} \cite{MR1976079,MR1976080,MR2177465,HaroEtAl2016}.
This parameterization $P$ defines a diffeomorphism, mapping a portion of the eigenspace onto the local invariant manifold.

Let $c \in \R^3$ be an equilibrium of $f$, and $\lambda = (\lambda_1, \dots, \lambda_d)$ the collection of all the $d \ge 1$ stable (resp. unstable) eigenvalues of $Df(c)$.
We denote by $\Lambda$ the $d$-by-$d$ diagonal matrix whose diagonal entries are given by $\lambda$.
The Hartman--Grobman Theorem \cite{Chicone} states that the linear and nonlinear dynamics near a hyperbolic equilibrium are topologically conjugate.
Motivated by this, we look for $P : U \subset \C^d \to \C^3$ as a \emph{conjugacy} between the (unknown) nonlinear flow $\varphi_t$ and the linear flow $e^{\Lambda t}$ restricted to the stable (resp. unstable) eigenspace:
\begin{equation}\label{eq:conjugacy_relation}
\varphi_t \circ P = P \circ e^{\Lambda t}, \qquad P(0) = c.
\end{equation}
Then, the infinitesimal version of the conjugacy relation~\eqref{eq:conjugacy_relation} is obtained by differentiating with respect to $t$ at $t=0$:
\begin{equation}\label{eq:invariance_equation}
DP(\theta) \Lambda \theta = f(P(\theta)).
\end{equation}
In words, the above \emph{invariance equation} means that $DP(\theta)$ maps the tangent vector $\Lambda \theta \in \C^d$ to a vector tangent to $P(U) \subset \C^3$.
The image $P(U)$ is the immersed local stable (resp. unstable) manifold.

\begin{remark}[Real image of the parameterization]\label{rem:real_image}
Consider the case $d = 2$.
For two real eigenvalues, the parameterization is real-valued on the real domain, namely $P : U \cap \R^2 \to \R^3$.
For a pair of complex conjugate eigenvalues,choosing the eigenvectors such that $v_2 = v_1^*$, the parameterization instead satisfies the conjugation symmetry
\begin{equation}
P(\theta_2^*, \theta_1^*) = P(\theta_1, \theta_2)^*,
\end{equation}
with the superscript $^*$ denoting complex conjugation.
Thus, the restriction of $P$ to $\big\{ (\rho e^{\mathrm{i} \alpha}, \rho e^{-\mathrm{i} \alpha}) : (\rho, \alpha) \in \R^2 \big\} \cap U$ is real-valued, and recovers the local invariant manifold.
\end{remark}

We consider a Taylor series expansion of $P = (P^{(1)}, P^{(2)}, P^{(3)})$, i.e.,
\begin{equation}
P(\theta)
= \sum_{k = (k_1,\dots,k_d) \in \N_0^d} P_k \theta^k, \qquad
\theta^k = \theta_1^{k_1} \cdots \theta_d^{k_d}, \quad P_k = (P_k^{(1)}, P_k^{(2)}, P_k^{(3)}),
\end{equation}
where $\N_0 \bydef \{0, 1, 2, \dots\}$.
There are two important observations which we give without proof (see, e.g., \cite{MR1976079,vanDenBergMirelesJamesReinhardt2016}):
\begin{enumerate}
\item
A condition for the solvability of~\eqref{eq:invariance_equation} is that the stable (resp. unstable) eigenvalues $\lambda_1, \dots, \lambda_d$ satisfy the \emph{non-resonance condition}
\begin{equation}
\lambda_i = k_1 \lambda_1 + \ldots + k_d \lambda_d \qquad \text{if and only if} \qquad k_j = \begin{cases}1, & j=i, \\ 0, & j \ne i.\end{cases}
\end{equation}
\item\label{it:scaling}
The scaling of the corresponding eigenvectors $v_1, \dots, v_d \in \C^3$ controls the size of the image of $P$ and its radius of convergence.
\end{enumerate}

In light of the above discussion, the \emph{manifold problem} involves the vector field $f$, its Jacobian $Df$, the parameters $a, b$, as well as the rigorously validated objects from Section~\ref{sec:eq_eig}: the equilibrium $c^\star$ with the associated collection of (un)stable eigenvalues $\lambda_1^\star, \dots, \lambda_d^\star$ and eigenvectors $v_1^\star, \dots, v_d^\star$.
We store all of these in a data structure.
\begin{juliacode}
struct ManifoldProblem
    f
    Df
    parameters
    equilibrium
    eigvals
    eigvecs
end
\end{juliacode}

\paragraph{Step 1: Defining the zero-finding problem.}
We observed in Section~\ref{sec:eigenspaces} that the stable eigenspace of $c_0^\star$ and the unstable eigenspace of $c_1^\star$ are two-dimensional, and so $d=2$.
Thus, we look for a bivariate Taylor series $P$, whose components $P^{(i)}$ belong to the space
\begin{equation}
X_{T,\nu}^{\otimes 2} \bydef
\left\{ u(\theta_1, \theta_2) = \sum_{k_1, k_2 \ge 0} u_{(k_1,k_2)} \theta_1^{k_1} \theta_2^{k_2} \, : \, u_{(k_1,k_2)} \in \mathbb{C}, \, \norm{u}_{X_{T,\nu}^{\otimes 2}} \bydef \sum_{k_1, k_2 \ge 0} | u_{(k_1,k_2)} | \nu^{k_1 + k_2 } < \infty  \right\}.
\end{equation}
Note that power series in that space are analytic inside the \emph{polydisk of radius $\nu$}
\begin{equation}
\mathbb{D}_\nu^2 \bydef \left\{ (z_1,z_2) \in \C^2 \, : \, |z_1| < \nu, \, |z_2| < \nu \right\}.
\end{equation}

Substituting the power series into~\eqref{eq:invariance_equation}, we obtain the infinite system of equations
\begin{equation}\label{eq:set_eq}
\begin{aligned}
\text{(zero order)} \qquad & f(P_{(0,0)}) = 0, \\
\text{(first order)} \qquad
& Df(P_{(0,0)})\, P_{(1,0)} = \lambda_1 P_{(1,0)}, \\
& Df(P_{(0,0)})\, P_{(0,1)} = \lambda_2 P_{(0,1)}, \\
\text{(higher order)} \qquad
& (k \cdot \lambda)\, P_k = [f(P)]_k, \qquad k_1 + k_2 \ge 2.
\end{aligned}
\end{equation}
As expected, $P_{(0,0)}$ corresponds to the equilibrium, while $P_{(1,0)}, P_{(0,1)}$ are the two stable (resp. unstable) eigenvectors.
At first order, the equation simply expresses that the stable (resp. unstable) eigenspace is tangent to the stable (resp. unstable) manifold at the equilibrium.

Consider the Banach space
\begin{equation}
\mathcal{X}_{T,\nu} \bydef \left( X_{T,\nu}^{\otimes 2} \right)^3, \qquad
\norm{P}_{\mathcal{X}_{T,\nu}} \bydef \sum_{i = 1}^3 \norm{P^{(i)}}_{X_{T,\nu}^{\otimes 2}}, \quad \text{for all } P = (P^{(1)}, P^{(2)}, P^{(3)}) \in \mathcal{X}_{T,\nu} .
\end{equation}
%
We can then rewrite the set of equations~\eqref{eq:set_eq} as the zero-finding problem $F : \mathcal{X}_{T,\nu} \to \mathcal{X}_{T,\nu}$ given by
\begin{equation}\label{eq:zero-finding-manifold}
F(P) \bydef
P - \phi - \mathrm{Diag}(\mathcal{L}_{T}) f(P),
\end{equation}
where, using the validated computation of the equilibrium $c^\star$ and (un)stable eigenpairs $(v_1^\star, \lambda_1^\star),(v_2^\star, \lambda_2^\star)$ from Section~\ref{sec:eq_eig},
\begin{equation}
\phi(\theta) = c^\star + v_1^\star \theta_1 + v_2^\star \theta_2,
\end{equation}
and,
\begin{equation}\label{eq:L_T}
\mathrm{Diag}( \mathcal{L}_{T} )
\bydef
\begin{pmatrix}
\mathcal{L}_{T} & 0 & 0 \\
0 & \mathcal{L}_{T} & 0 \\
0 & 0 & \mathcal{L}_{T}
\end{pmatrix}, \qquad
(\mathcal{L}_{T})_{k,l}
\bydef
\begin{cases}
(k_1 \lambda_1^\star + k_2 \lambda_2^\star)^{-1}, & k = l, \, k_1 + k_2 \ge 2, \\
0, & \text{otherwise}.
\end{cases}
\end{equation}
Note that $k_1 \lambda_1^\star + k_2 \lambda_2^\star \ne 0$, since $\lambda_1^\star, \lambda_2^\star$ lie in the same half-plane.

Using the \texttt{RadiiPolynomial} library, we can store the coefficients of an expansion in a prescribed basis (Taylor, Fourier, Chebyshev, $\ldots$) as a \texttt{Sequence}; here the relevant basis is the tensor product \texttt{TensorSpace\{NTuple\{2,Taylor\}\}} of two Taylor bases.
Linear operators are organized by a hierarchy of abstract types reflecting their structure.
We implement $\mathcal{L}_T$ as a subtype of \texttt{AbstractDiagonalOperator} which tells the library that the operator is diagonal (matching domain and codomain), and all that remains is to specify its entries by overloading \texttt{getcoefficient}.
Each entry is indexed by a pair \texttt{(codom, k)}, \texttt{(dom, l)} of a space and a multi-index; following~\eqref{eq:L_T}, the method returns $(k_1 \lambda_1 + k_2 \lambda_2)^{-1}$ when $k = l$ and $k_1 + k_2 \ge 2$, and zero otherwise.
\begin{juliacode}
struct L_manifold <: AbstractDiagonalOperator
    λ₁
    λ₂
end

function RadiiPolynomial.getcoefficient(L::L_manifold, (codom, k)::T, (dom, l)::T) where {T<:Tuple{TensorSpace{NTuple{2,Taylor}},NTuple{2,Int}}}
    x = inv(exact(k[1]) * L.λ₁ + exact(k[2]) * L.λ₂)
    k == l && sum(k) ≥ 2 && return x
    return zero(x)
end
\end{juliacode}
Furthermore, we note that $X_{T,\nu}^{\otimes 2}$ comes naturally equipped with a multiplication operation, the \emph{Cauchy product}, so that for any $P, Q \in X_{T,\nu}^{\otimes 2}$, we have
\begin{equation}\label{eq:cauchy_product}
\begin{aligned}
&P(\theta) Q(\theta) = (P * Q)(\theta) = \sum_{k_1,k_2 \ge 0} (P * Q)_{(k_1,k_2)} \theta_1^{k_1}\theta_2^{k_2}, \\
&\text{where } (P * Q)_{(k_1,k_2)} \bydef \sum_{l_1=0}^{k_1} \sum_{l_2=0}^{k_2} P_{(k_1 - l_1, k_2 - l_2)} Q_{(l_1,l_2)}.
\end{aligned}
\end{equation}
In fact, $(X_{T,\nu}^{\otimes 2}, *)$ forms a Banach algebra as $\norm{P * Q}_{X_{T,\nu}^{\otimes 2}} \le \norm{P}_{X_{T,\nu}^{\otimes 2}} \norm{Q}_{X_{T,\nu}^{\otimes 2}}$.
Hence, the vector field $f : \R^3 \to \R^3$ naturally extends as a map acting on $\mathcal{X}_{T,\nu}$, and we use the same notation to denote this map $f : \mathcal{X}_{T,\nu} \to \mathcal{X}_{T,\nu}$, so that
\begin{equation}
f(P) =
\begin{pmatrix}
P^{(2)} \\
P^{(1)} - a P^{(2)} - P^{(1)} * P^{(3)} \\
-b P^{(3)} + P^{(1)} * P^{(1)}
\end{pmatrix}, \qquad \text{for all } P = (P^{(1)}, P^{(2)}, P^{(3)}) \in \mathcal{X}_{T,\nu}.
\end{equation}
\begin{juliacode}
function F(pb::ManifoldProblem, u)
    L = L_manifold(pb.eigvals[1], pb.eigvals[2])

    θ₁ = Sequence(Taylor(1) ⊗ Taylor(0), [exact(0), exact(1)])
    θ₂ = Sequence(Taylor(0) ⊗ Taylor(1), [exact(0), exact(1)])
    ϕ = pb.equilibrium + pb.eigvecs[:,1] .* θ₁ + pb.eigvecs[:,2] .* θ₂

    return u - ϕ - L .* pb.f(u, pb.parameters)
end

function DF(pb::ManifoldProblem, u)
    L = L_manifold(pb.eigvals[1], pb.eigvals[2])

    return Diagonal([exact(I), exact(I), exact(I)]) - L .* Multiplication.(pb.Df(u, pb.parameters))
end
\end{juliacode}

\paragraph{Step 2: Computing the approximate zero (with floating-point arithmetic).}
We wish to obtain a polynomial approximation of order $K\ge 1$ of the manifold
\begin{equation}
\bar{P}(\theta) = \sum_{k_1 = 0}^K \sum_{k_2 = 0}^K \bar{P}_{(k_1,k_2)} \theta_1^{k_1} \theta_2^{k_2},
\end{equation}
Formally, we introduce the projection operator $\Pi_{\le K} : X_{T,\nu}^{\otimes 2} \to X_{T,\nu}^{\otimes 2}$ defined by
\begin{equation}
(\Pi_{\le K} P)_k \bydef
\begin{cases}
P_k, & \max(k_1, k_2) \le K, \\
0, & \max(k_1, k_2) > K,
\end{cases} \qquad k=(k_1,k_2)\in \N_0^2.
\end{equation}
Its complement is denoted by $\Pi_{> K} \bydef I - \Pi_{\le K}$.
Moreover, we extend $\Pi_{\le K}$ to $\mathcal{X}_{T,\nu}$ by acting component-wise
\begin{equation}
\Pi_{\le K} P = (\Pi_{\le K} P^{(1)}, \Pi_{\le K} P^{(2)}, \Pi_{\le K} P^{(3)}),
\end{equation}
and similarly for $\Pi_{>K}$.

We use Newton's method on a finite approximation of the zero-finding problem~\eqref{eq:zero-finding-manifold} $\Pi_{\le K} \circ F \circ \Pi_{\le K}$, and where $c^\star$, $\lambda_1^\star, \lambda_2^\star$ and $v_1^\star, v_2^\star$ are all replaced by their floating-point approximations, denoted $\bar{c}$, $\bar{\lambda}_1, \bar{\lambda}_2$ and $\bar{v}_1, \bar{v}_2$.
The initial guess for Newton's method $P_{init}$ is a first order approximation of the manifold, i.e., $P_{init} \in \Pi_{\le 1} \mathcal{X}_{T,\nu}$ satisfies the zeroth and first order conditions in~\eqref{eq:set_eq},
\begin{subequations}
\begin{align}
P_{init} &= P_{init,(0,0)} + P_{init,(1,0)} \theta_1 + P_{init,(0,1)} \theta_2, \\
P_{init,(0,0)} &= \bar{c}, \qquad P_{init,(1,0)} = \bar{v}_1, \qquad P_{init,(0,1)} = \bar{v}_2.
\end{align}
\end{subequations}
As observed in Point~\ref{it:scaling}, the scaling of the eigenvectors controls the decay rate of the Taylor series, so by rescaling them one can ensure convergence of Newton's method.
\begin{juliacode}

function approximate_zero(pb_approx::ManifoldProblem, order)
    P_init = zeros(eltype(pb_approx.eigvals), (Taylor(order) ⊗ Taylor(order))^3)
    for j ∈ 1:3
        component(P_init, j)[(0,0)] = pb_approx.equilibrium[j]
        component(P_init, j)[(1,0)] = pb_approx.eigvecs[j,1]
        component(P_init, j)[(0,1)] = pb_approx.eigvecs[j,2]
    end

    F_DF_newton(P) = F(pb_approx, unpack(P)), DF(pb_approx, unpack(P))
    P_bar, newton_success = newton(F_DF_newton, P_init; verbose = true)
    return P_bar, newton_success
end
\end{juliacode}

\paragraph{Step 3: Constructing the approximate inverse (with floating-point arithmetic).}
This step is less straightforward than in the finite-dimensional case.
Indeed, simply inverting a finite truncation of $DF(\bar{P})$ is not sufficient, and we must exploit its structure.
The Fr\'echet derivative reads
\begin{equation}
DF(P) =
I - \mathrm{Diag}( \mathcal{L}_{T} ) Df(P).
\end{equation}

Note that $Df(P)$ is a 3-by-3 block operator here, whose entries are multiplication operators.
Since the product is the Cauchy product given in~\eqref{eq:cauchy_product}, each entry $\partial_{P^{(j)}} f^{(i)}(P)$ corresponds to a lower triangular Toeplitz operator.
Moreover, $\mathcal{L}_{T} : X_{T,\nu}^{\otimes 2} \to X_{T,\nu}^{\otimes 2}$ is a compact operator and its tail $\Pi_{>K}\mathcal{L}_T$ has small operator norm for sufficiently large truncation order $K$.
Composing on the left with $\mathcal{L}_T$ then dampens rows beyond $k = (k_1, k_2)$ with $\max(k_1, k_2) > K$, so that $\mathcal{L}_{T} \circ \partial_{P^{(j)}} f^{(i)}(P)$ has negligible entries in those rows.
Visually,
\begin{equation}
\partial_{P^{(j)}} f^{(i)}(P) =
\begin{tikzpicture}[baseline=(m.center)]
\matrix (m) [matrix of math nodes,
             nodes in empty cells,
             nodes={minimum width=1em, minimum height=1em},
             left delimiter=(,
             right delimiter=),
             row sep=0.6em,
             column sep=0.6em] {
  &   &   &   \\
  &   & 0 &   \\
  &   &   &   \\
  &   &   &   \\
};

\fill[gray!20]
(m-4-1.south west) --
(m-4-4.south east) --
(m-1-1.north west) --
cycle;

\draw[thick]
(m-1-1.north west) -- (m-4-4.south east);

\end{tikzpicture}, \quad
\mathcal{L}_{T} \circ \partial_{P^{(j)}} f^{(i)}(P)
\approx
\begin{tikzpicture}[baseline=(m.center)]
\matrix (m) [matrix of math nodes,
             nodes in empty cells,
             nodes={minimum width=1em, minimum height=1em},
             left delimiter=(,
             right delimiter=),
             row sep=0.6em,
             column sep=0.6em] {
  &   &   &   \\
  &   & 0 &   \\
  &   &   &   \\
  & 0 &   &   \\
};

\fill[gray!20]
(m-3-1.north west) --
(m-3-3.north west) --
(m-1-1.north west) --
cycle;

\draw[dashed]
(m-3-1.north west) -- (m-3-3.north west);

\draw[thick]
(m-1-1.north west) -- (m-3-3.north west);
\draw[dashed]
(m-3-3.north west) -- (m-4-4.south east);

\draw[decorate,decoration={brace,mirror}]
($(m-3-1.north west)+(-1.5em,0)$) --
($(m-4-1.south west)+(-1.5em,0)$)
node[midway,xshift=-1.5em] {$\Pi_{>K}$};

\end{tikzpicture}.
\end{equation}
This logic applies to each of the blocks composing $DF(\bar{P})$.
Hence, for sufficiently large $K$, the term $\mathrm{Diag}(\mathcal{L}_{T}) Df(\bar{P})$ is well-approximated by a finite truncation, so
\begin{equation}
DF(\bar{P}) \approx I - \Pi_{\le K} \Big( \mathrm{Diag}( \mathcal{L}_{T} ) Df(\bar{P}) \Big) \Pi_{\le K}.
\end{equation}
This approximation of $DF(\bar{P})$ is far easier to invert: letting $A_{\le K} \approx \left( \Pi_{\le K} DF(\bar{P}) \Pi_{\le K}  \right)^{-1}$ (i.e., the inverse of a $3(K+1)^2$-by-$3(K+1)^2$ matrix), we set
\begin{equation}\label{eq:A_manifold}
A = A_{\le K} \Pi_{\le K} + \Pi_{> K}.
\end{equation}

\begin{juliacode}
function approximate_inverse(pb_approx::ManifoldProblem, P_bar)
    Π = Projection(space(P_bar))
    A_finite = inv(mid(Π * DF(pb_approx, unpack(P_bar)) * Π))
    A = interval(A_finite) + (interval(I) - interval(Π))
    return A
end
\end{juliacode}

\begin{remark}
That we chose $K \ge 1$ to be the same for $\bar{P}$ and $A$ is only for convenience and to avoid introducing too many symbols; but in principle, these may be chosen independently.
\end{remark}

\paragraph{Step 4: Estimating the bounds (with interval arithmetic).}
The first observation is that the $Y$ bound consists only of a finite number of calculations and is therefore directly computable.
Indeed, $\bar{P} \in \Pi_{\le K} \mathcal{X}_{T,\nu}$, and since $\mathcal{L}_T$ is diagonal and $f$ is quadratic, $F(\bar{P}) \in  \Pi_{\le 2K} \mathcal{X}_{T,\nu}$.
By construction, $A$ preserves this truncation order, so $AF(\bar{P}) \in  \Pi_{\le 2K} \mathcal{X}_{T,\nu}$, and therefore
\begin{equation}
\norm{ A F(\bar{P}) }_{\mathcal{X}_{T,\nu}} = \sum_{i = 1}^3 \norm{\sum_{j = 1}^3 \Pi_{\le 2K} A^{(i,j)} F^{(j)}(\bar{P})}_{X_{T,\nu}^{\otimes 2}}.
\end{equation}

Since the $Z(R)$ bound in the Radii Polynomial Theorem~\ref{thm:rpa} consists in estimating the supremum of the operator norm of $I - A DF(P)$ over the closed ball $B(\bar{P}, R) \subset \mathcal{X}_{T,\nu}$, it requires some more analysis to derive formulas that can be estimated by a computer.

\begin{pproposition}[Formula for the $Z$ bound]\label{prop:rpa_manifold}
Let $K \ge 1$ and $\bar{P} \in \Pi_{\le K} \mathcal{X}_{T,\nu}$, and $\lambda_1^\star, \lambda_2^\star \in \C$ such that $\mathrm{Re}(\lambda_1^\star) \mathrm{Re}(\lambda_2^\star) > 0$, i.e., both eigenvalues lie in the same half-plane.
Consider the map $F$ given in~\eqref{eq:zero-finding-manifold} and $A$ given in~\eqref{eq:A_manifold}.
Consider $R \in [0,\infty]$, and $Z_0, Z_1(R) \ge 0$ satisfying
\begin{subequations}\label{eq:Z_bound_ivp}
\begin{align}
\begin{aligned}
\max\Big( \norm{\Pi_{\le K} - \Pi_{\le 2K} A \Pi_{\le 2K} DF(\bar{P}) \Pi_{\le K}}_{\mathscr{B}(\mathcal{X}_{T,\nu})}&, \\
\norm{\Pi_{> K} \mathcal{L}_T \Pi_{> K}}_{\mathscr{B}(X_{T,\nu}^{\otimes 2})} \norm{D f(\bar{P})}_{\mathscr{B}(\mathcal{X}_{T,\nu})}& \Big)
\end{aligned} &\le Z_0, \\
2R \norm{A \mathrm{Diag}(\mathcal{L}_{T})}_{\mathscr{B}(\mathcal{X}_{T,\nu})} &\le Z_1(R),
\end{align}
\end{subequations}
where
\begin{subequations}
\begin{align}
\norm{\Pi_{> K} \mathcal{L}_T \Pi_{> K}}_{\mathscr{B}(X_{T,\nu}^{\otimes 2})}
&\le \frac{1}{(K+1)\min_{i=1,2} |\mathrm{Re}(\lambda_i^\star)|}, \\
\norm{A \mathrm{Diag}(\mathcal{L}_{T})}_{\mathscr{B}(\mathcal{X}_{T,\nu})}
&\le \max\left(\norm{A_{\le K} \mathrm{Diag}(\mathcal{L}_{T}) \Pi_{\le K}}_{\mathscr{B}(\mathcal{X}_{T,\nu})}, \norm{\Pi_{> K} \mathcal{L}_T \Pi_{> K}}_{\mathscr{B}(X_{T,\nu}^{\otimes 2})} \right).
\end{align}
\end{subequations}
Then,
\begin{equation}
\sup_{P \in B(\bar{P}, R)} \norm{I - A DF(P)}_{\mathscr{B}(\mathcal{X}_{T,\nu})} \le Z(R) = Z_0 + Z_1(R).
\end{equation}
\end{pproposition}

\begin{proof}
Let $P \in B(\bar{P}, R)$.
The triangle inequality yields
\begin{equation*}
\norm{I - A DF(P)}_{\mathscr{B}(\mathcal{X}_{T,\nu})}
\le \norm{I - A DF(\bar{P})}_{\mathscr{B}(\mathcal{X}_{T,\nu})} + \norm{A(DF(P) - DF(\bar{P}))}_{\mathscr{B}(\mathcal{X}_{T,\nu})}.
\end{equation*}
The goal is to show that the two terms are bounded by $Z_0$ and $Z_1(R)$, respectively.
The argument relies on two ingredients.
First, since $\mathrm{Re}(\lambda_1^\star) \mathrm{Re}(\lambda_2^\star) > 0$, we have $|k_1\lambda_1^\star + k_2 \lambda_2^\star| \ge (k_1 + k_2)\min_{i=1,2} |\mathrm{Re}(\lambda_i^\star)|$, so the tail bound of $\mathcal{L}_T$ satisfies
\begin{equation*}
\norm{\Pi_{>K} \mathcal{L}_{T} \Pi_{>K}}_{\mathscr{B}(X^{\otimes 2}_{T,\nu})}
= \sup_{\max(k_1,k_2) > K} \frac{1}{|k_1\lambda_1^\star + k_2 \lambda_2^\star|}
\le \frac{1}{(K+1)\min_{i=1,2}|\mathrm{Re}(\lambda_i^\star)|}.
\end{equation*}
Second, the operator norm on $\mathcal{X}_{T,\nu}$ decomposes as
\begin{equation*}
\norm{L}_{\mathscr{B}(\mathcal{X}_{T,\nu})}
= \max\left( \norm{L \Pi_{\le K}}_{\mathscr{B}(\mathcal{X}_{T,\nu})}, \norm{L \Pi_{> K}}_{\mathscr{B}(\mathcal{X}_{T,\nu})} \right), \qquad \text{for all } L \in \mathscr{B}(\mathcal{X}_{T,\nu}).
\end{equation*}

For the $Z_1(R)$ bound, since $DF(P) - DF(\bar{P}) = -\mathrm{Diag}(\mathcal{L}_T)(Df(P) - Df(\bar{P}))$ and the only entries of
\begin{equation*}
Df(P) =
\begin{pmatrix}
0 &     1  &   0 \\
1 - P^{(3)} &  -a &  -P^{(1)} \\
2P^{(1)}  &    0  & -b
\end{pmatrix}
\end{equation*}
that depend on $P$ are at positions $(2,1)$, $(2,3)$, and $(3,1)$, a direct estimate yields
\begin{equation*}
\norm{Df(P) - Df(\bar{P})}_{\mathscr{B}(\mathcal{X}_{T,\nu})} \le 2R.
\end{equation*}
Applying the operator norm decomposition above with $L = A \mathrm{Diag}(\mathcal{L}_T)$ and using $A \Pi_{>K} = \Pi_{>K}$ together with the tail bound,
\begin{equation*}
\norm{A \mathrm{Diag}(\mathcal{L}_T)}_{\mathscr{B}(\mathcal{X}_{T,\nu})} \le
\max\left( \norm{A_{\le K} \mathrm{Diag}(\mathcal{L}_T) \Pi_{\le K}}_{\mathscr{B}(\mathcal{X}_{T,\nu})}, \frac{1}{(K+1)\min_{i=1,2}|\mathrm{Re}(\lambda_i^\star)|} \right).
\end{equation*}
Hence, $\norm{A(DF(P) - DF(\bar{P}))}_{\mathscr{B}(\mathcal{X}_{T,\nu})} \le 2R \norm{A \mathrm{Diag}(\mathcal{L}_T)}_{\mathscr{B}(\mathcal{X}_{T,\nu})} \le Z_1(R)$ as desired.

For the $Z_0$ bound, applying the operator norm decomposition above with $L = I - A DF(\bar{P})$ reduces the problem to estimating $\big[ I - A DF(\bar{P}) \big] \Pi_{\le K}$ and $\big[ I - A DF(\bar{P}) \big] \Pi_{> K}$ separately.
Using $A \Pi_{>K} = \Pi_{>K}$, that $\mathcal{L}_T$ is diagonal, and that $Df(\bar{P}) \Pi_{>K} = \Pi_{>K} Df(\bar{P}) \Pi_{>K}$ (from the property of the Cauchy product, which yields lower triangular multiplication operators), we obtain
\begin{equation*}
\big[ I - A DF(\bar{P}) \big] \Pi_{> K}
= \Pi_{> K} \mathrm{Diag}(\mathcal{L}_T) \Pi_{> K} Df(\bar{P}) \Pi_{> K},
\end{equation*}
whose norm is at most $\frac{1}{(K+1)\min_{i=1,2}|\mathrm{Re}(\lambda_i^\star)|} \norm{Df(\bar{P})}_{\mathscr{B}(\mathcal{X}_{T,\nu})}$ by submultiplicativity and the tail bound of $\mathcal{L}_T$ established above.
On the truncated part, since $\bar{P} \in \Pi_{\le K} \mathcal{X}_{T,\nu}$ and $f$ is quadratic, $DF(\bar{P}) \Pi_{\le K} = \Pi_{\le 2K} DF(\bar{P}) \Pi_{\le K}$.
Hence, $A DF(\bar{P}) \Pi_{\le K} = \Pi_{\le 2K} A \Pi_{\le 2K} DF(\bar{P}) \Pi_{\le K}$ is a finite matrix, and
\begin{equation*}
\norm{ \big[ I - A DF(\bar{P}) \big] \Pi_{\le K}}_{\mathscr{B}(\mathcal{X}_{T,\nu})}
= \norm{ \Pi_{\le K} - \Pi_{\le 2K} A \Pi_{\le 2K} DF(\bar{P}) \Pi_{\le K}}_{\mathscr{B}(\mathcal{X}_{T,\nu})}.
\end{equation*}
Combining these shows that $\norm{I - A DF(\bar{P})}_{\mathscr{B}(\mathcal{X}_{T,\nu})} \le Z_0$.
\end{proof}

We stress that, since $A$ is a finite-rank perturbation of the identity, its surjectivity is equivalent to its injectivity; hence, verifying $Z(R) = Z_0 + Z_1(R) < 1$ is enough to imply that $A$ is injective.
Then, provided moreover that $Y/(1-Z(R)) \le R$, the Radii Polynomial Theorem~\ref{thm:rpa} yields a rigorous a posteriori error bound $r = Y/(1 - Z)$, with respect to the norm $\norm{\,\cdot\,}_{\mathcal{X}_{T,\nu}}$, on the polynomial approximation $\bar{P}$ of the parameterization of the local invariant manifold.
In particular, the mathematically exact parameterization $P^\star$ can be written as
\begin{equation}
P^\star = \bar{P} + h, \qquad \| h \|_{\mathcal{X}_{T,\nu}} \le \frac{Y}{1-Z}.
\end{equation}
The validation above can be carried out for any $\nu > 0$; in view of Section~\ref{sec:heteroclinic}, we perform it with $\nu > 1$, so that $P^\star$ is analytic on a polydisk containing the unit torus $|\theta_1| = |\theta_2| = 1$.
Moreover, for the complex conjugate eigenpair of $c_1^\star$, we take $v_2^\star = (v_1^\star)^*$ and choose $\bar{P}$ satisfying the conjugation symmetry of Remark~\ref{rem:real_image}.
Then the ball $B(\bar{P}, r)$ is invariant under the map $S : P_{(k_1,k_2)} \mapsto P_{(k_2,k_1)}^*$ and $F$ is equivariant with respect to it, i.e., $F(SP) = SF(P)$, so the uniqueness of the zero of $F$ in $B(\bar{P}, r)$ implies that the exact $P^\star$ satisfies the symmetry as well; in particular, $P^\star(\rho e^{\mathrm{i}\alpha}, \rho e^{-\mathrm{i}\alpha}) \in \R^3$.

We carried out this validation for the stable manifold of $c_0^\star$ and the unstable manifold of $c_1^\star$; the truncation orders, $\nu$ and the eigenvector scalings are given in~\cite{Code}.
\begin{juliacode}
Y_bound(pb::ManifoldProblem, P_bar, A, X) = norm(A * F(pb, unpack(P_bar)), X)

function Z_bound(pb::ManifoldProblem, P_bar, A, R, X_T, X)
    K = order(component(P_bar, 1))[1]
    Π = interval(Projection(space(P_bar)))

    L = Diagonal([L_manifold(pb.eigvals[1], pb.eigvals[2]) for j ∈ 1:3])
    tail_bound_L = inv(interval(K + 1) * minimum(abs ∘ real, pb.eigvals))
    opnorm_AL = max(opnorm(A * (L * Π), X), tail_bound_L)

    B = DF(pb, unpack(P_bar))
    W = pb.Df(unpack(P_bar), pb.parameters)
    Z₀ = max(opnorm(Π - A * (B * Π), X),
             tail_bound_L * opnorm(norm.(W, X_T), 1))

    Z₁ = exact(2) * R * opnorm_AL

    return Z₀ + Z₁
end
\end{juliacode}
%


\section{Validated computation of transverse intersection}\label{sec:heteroclinic}

To not overburden the article, we do not provide the code alongside the mathematical discussion contained in this section; the code, however, is available in full at \cite{Code}.

We now have all the necessary pieces to prove the existence of a heteroclinic orbit connecting the two equilibria $c_0^\star, c_1^\star$, that is, we seek a solution to the system~\eqref{eq:shimizu_morioka} satisfying
\begin{equation}
\lim_{t \to -\infty} w(t) = c_1^\star, \qquad \lim_{t \to +\infty} w(t) = c_0^\star.
\end{equation}
By restricting to a finite time interval $[0,\tau]$, this can be reformulated as the boundary-value problem
\begin{equation}\label{eq:bvp_heteroclinic}
\begin{cases}
\displaystyle \frac{\mathrm{d}}{\mathrm{d}t} w(t) = f(w(t)), & t \in [0,\tau], \\
w(0) \in W^u_{\mathrm{loc}}(c_1^\star), \\
w(\tau) \in W^s_{\mathrm{loc}}(c_0^\star).
\end{cases}
\end{equation}

In the previous section, rigorous parameterizations of the local invariant manifolds have been computed, for some $\nu > 1$,
\begin{equation}
P^\star, Q^\star : \mathbb{D}_\nu^2 \to \mathbb{C}^3,
\end{equation}
representing the local unstable manifold of $c_1^\star$ and the local stable manifold of $c_0^\star$, respectively.
The boundary conditions can then be written as
\begin{equation}\label{eq:boundary_conditions}
w(0) = P^\star(\sigma),
\qquad
w(\tau) = Q^\star(\theta),
\end{equation}
where $\sigma = (\rho e^{\mathrm{i}\alpha}, \rho e^{-\mathrm{i}\alpha})$ with $0 \le \rho < \nu$, $\alpha \in \R$, and $\theta \in \mathbb{D}_\nu^2 \cap \mathbb{R}^2$ are local manifold coordinates (recall from Remark~\ref{rem:real_image} that then $P^\star(\sigma), Q^\star(\theta) \in \mathbb{R}^3$).

\paragraph{Step 1: Defining the zero-finding problem.}
Beyond the trajectory $t \mapsto w(t)$, the unknowns are the local coordinates $\sigma = (\rho e^{\mathrm{i}\alpha}, \rho e^{-\mathrm{i}\alpha})$, $0 \le \rho < \nu$, $\alpha \in \R$, on the unstable side and $\theta \in \mathbb{D}_\nu^2 \cap \R^2$ on the stable side, together with the integration time $\tau > 0$, totaling 5 scalar parameters.
The boundary condition $w(\tau) = Q^\star(\theta)$ contributes 3 scalar equations, leaving 2 degrees of freedom to fix.
We do so by prescribing $\tau > 0$ and restricting the unstable coordinates to the unit circle $\rho=1$, so that
\begin{equation}
\sigma = \gamma(\alpha) \bydef \left(e^{\mathrm{i}\alpha}, e^{-\mathrm{i}\alpha}\right), \qquad \alpha \in \R.
\end{equation}
Since $\nu > 1$ was used in Section~\ref{sec:manifold}, we do have that $\gamma(\alpha)$ lies inside the domain of analyticity $\mathbb{D}_\nu^2$ of the parameterization $P^\star$ of the local unstable manifold of $c^\star_1$.

To represent the trajectory $t \mapsto w(t)$, we expand it in Chebyshev polynomials of the first kind, $T_k(s) \bydef \cos(k \arccos s)$, $s \in [-1, 1]$.
Unlike Taylor series, Chebyshev expansions are well-suited for non-local trajectories.
In particular, an analytic function on $[-1,1]$ admits a Chebyshev series whose coefficients decay exponentially fast, and therefore belongs, for some $\mu \ge 1$, to the Banach space
\begin{equation}
X_{C,\mu} \bydef \left\{ u(s) = u_0 + 2 \sum_{k \ge 1} u_k T_k(s) \, : \, u_k \in \mathbb{R},\, \norm{u}_{X_{C,\mu}} \bydef | u_0 | + 2 \sum_{k \ge 1} | u_k | \mu^k < \infty \right\}.
\end{equation}
For $\mu > 1$, series in $X_{C,\mu}$ extend analytically inside the Bernstein ellipse
\begin{equation}
\mathbb{E}_\mu \bydef \left\{ z \in \C \, : \, z = \frac{1}{2}(w + w^{-1}), \, 1 \le |w| < \mu\right\}.
\end{equation}
We refer to \cite{Trefethen} for a thorough exposition.

We consider the time rescaling $s \in [-1,1] \mapsto t(s) = \tau(s + 1)/2$, and introduce the rescaled trajectory $u(s) = w(t(s))$.
Using the boundary conditions~\eqref{eq:boundary_conditions} and integrating~\eqref{eq:bvp_heteroclinic} from $-1$ to $s$, we obtain
\begin{equation}
\begin{cases}
\displaystyle u(s) = P^\star(\gamma(\alpha)) + \frac{\tau}{2} \int_{-1}^s f(u(s')) \, \mathrm{d} s', & s \in [-1,1], \\
u(1) = Q^\star(\theta).
\end{cases}
\end{equation}
Let
\begin{equation}
\mathcal{X}_{C,\mu} \bydef X_{C,\mu}^3 \times \R^3, \qquad
\norm{x}_{\mathcal{X}_{C,\mu}} \bydef \sum_{i = 1}^3 \norm{u^{(i)}}_{X_{C,\mu}} + |\alpha| + |\theta_1| + |\theta_2|,
\end{equation}
for all $x = (u, \alpha, \theta_1, \theta_2) \in \mathcal{X}_{C,\mu}$ with $u = (u^{(1)}, u^{(2)}, u^{(3)}) \in X_{C,\mu}^3$ and $\alpha, \theta_1, \theta_2 \in \R$.
Then, a heteroclinic orbit can be seen as a zero of the map $F : \mathcal{X}_{C,\mu} \to \mathcal{X}_{C,\mu}$ given by
\begin{equation}\label{eq:zero-finding-heteroclinic}
F(u, \alpha, \theta_1, \theta_2) \bydef
\begin{pmatrix}
\displaystyle u - P^\star (\gamma(\alpha)) - \frac{\tau}{2} \mathrm{Diag}( \mathcal{L}_C ) f(u) \\
u(1) - Q^\star (\theta_1, \theta_2)
\end{pmatrix},
\end{equation}
where
\begin{equation}\label{eq:Lc}
\mathrm{Diag}( \mathcal{L}_C ) \bydef
\begin{pmatrix}
\mathcal{L}_C & 0 & 0 \\
0 & \mathcal{L}_C & 0 \\
0 & 0 & \mathcal{L}_C
\end{pmatrix},
\qquad
(\mathcal{L}_C u )_k \bydef
\begin{cases}
\displaystyle u_0 - \frac{u_1}{2} + 2\sum_{l \ge 2} \frac{(-1)^{l+1}}{l^2 - 1} u_l, & k = 0, \\
\displaystyle \frac{u_{k-1} - u_{k+1}}{2k}, & k \ge 1.
\end{cases}
\end{equation}
As with Taylor series (see Section~\ref{sec:manifold}), we note that $X_{C,\mu}$ comes naturally equipped with a multiplication operation, the \emph{discrete convolution}, so that for any $u, w \in X_{C,\mu}$, we have
\begin{equation}\label{eq:cheb_convolution}
\begin{aligned}
&u(s) w(s) = (u * w)(s) = (u * w)_{0} + 2 \sum_{k \ge 1} (u * w)_{k} T_k(s), \\
&\text{where } (u * w)_{k} \bydef \sum_{l \in \Z} u_{|k - l|} w_{|l|}, \qquad k \ge 0.
\end{aligned}
\end{equation}
In addition, $(X_{C,\mu}, *)$ is a Banach algebra such that $\norm{u * w}_{X_{C,\mu}} \le \norm{u}_{X_{C,\mu}} \norm{w}_{X_{C,\mu}}$.
Once more, the vector field $f : \R^3 \to \R^3$ can be extended as a map $f : X_{C,\mu}^3 \to X_{C,\mu}^3$, where the arithmetic operations (addition and multiplication) should be understood as those of $X_{C,\mu}$.

\begin{remark}[Notation]
We purposely use the same notation $*$ for the Cauchy product and the discrete convolution, as well as for the vector field as a map on $\R^3$, $\mathcal{X}_{T,\nu}$ and $X_{C,\mu}^3$.
The reason is that they all have the same meaning (similarly to how $+$ denotes the additions of real numbers and that of vectors); moreover, they are unambiguously interpreted from the context of the operand.
\end{remark}

\begin{remark}[Transversality]
We stress that the transversality of the intersection is a consequence of verifying the contraction in the Radii Polynomial Theorem~\ref{thm:rpa}.
This fact relates to the local uniqueness and invertibility of $DF$ at the zero $x^\star$ of $F$; see, e.g., \cite{LessardMirelesJamesReinhardt2014}.
\end{remark}

To use the map $F$ and its Fréchet derivative $DF$, we need to know how to rigorously evaluate the parameterizations $P^\star, Q^\star \in \mathcal{X}_{T,\nu}$ and their derivatives.
The following lemma is applied to each component.

\begin{plemma}[Formula for rigorous evaluation]\label{lem:rigorous_eval}
Let $\nu > 1$, $P^\star \in X_{T,\nu}^{\otimes 2}$, $\bar{P} \in \Pi_{\le K} X_{T,\nu}^{\otimes 2}$ such that $\norm{P^\star - \bar{P}}_{X_{T,\nu}^{\otimes 2}} \le r$.
For $\theta = (\theta_1, \theta_2) \in \mathbb{D}_\nu^2$, it holds that
\begin{subequations}\label{eq:rigorous_eval}
\begin{align}
|P^\star(\theta_1, \theta_2) - \bar{P}(\theta_1, \theta_2)|
&\le r, \\
|\partial_{\theta_j} P^\star(\theta_1, \theta_2) - \partial_{\theta_j} \bar{P}(\theta_1, \theta_2)|
&\le r\beta_\nu(|\theta_j|), \qquad j=1,2,
\end{align}
\end{subequations}
where
\begin{equation}\label{eq:beta_nu}
\beta_\nu(\delta) \bydef
\begin{cases}
\nu^{-1}, & 0 \le \delta \le \nu/e, \\[2pt]
\dfrac{1}{e \delta |\ln(\delta/\nu)|}, & \nu/e < \delta < \nu.
\end{cases}
\end{equation}
\end{plemma}

\begin{proof}
The first inequality follows from the fact that the $X_{T,\nu}^{\otimes 2}$-norm controls the supremum norm on the polydisk $\mathbb{D}_\nu^2$.
For the second inequality, denoting $h = P^\star - \bar{P}$,
\begin{equation*}
|\partial_{\theta_1} h(\theta_1, \theta_2)|
\le \sum_{k_1\ge 1,\, k_2 \ge 0} k_1 |h_{(k_1,k_2)}| \, |\theta_1|^{k_1-1} |\theta_2|^{k_2}
\le r \sup_{k_1 \ge 1} k_1 |\theta_1|^{k_1-1} \nu^{-k_1}.
\end{equation*}
Let $\delta = |\theta_1| \in (0, \nu)$, the continuous extension $g(s) \bydef s \delta^{-1} (\delta/\nu)^s$ has $g'(s) = \delta^{-1} (\delta/\nu)^s (1 + s \ln(\delta/\nu))$, which vanishes at $s_\mathrm{crit} = -1/\ln(\delta/\nu) > 0$.
Hence $\sup_{k_1 \ge 1} g(k_1) \le g(s_\mathrm{crit}) = (e \delta |\ln(\delta/\nu)|)^{-1}$ when $\delta > \nu/e$, i.e., $s_\mathrm{crit} > 1$; otherwise $g$ is decreasing on $[1, \infty)$ and $\sup_{k_1 \ge 1} g(k_1) = g(1) = \nu^{-1}$.
For $\delta = 0$, only the terms with $k_1 = 1$ remains in the sum, which again yields the bound $r \nu^{-1}$.
\end{proof}

In other words, the two inequalities in~\eqref{eq:rigorous_eval} indicate that the evaluation of $P^\star$ on $\sigma = \gamma(\alpha) = (e^{\mathrm{i}\alpha}, e^{-\mathrm{i}\alpha})$ and $Q^\star$ on $\theta \in \mathbb{R}^2$ (both coordinates lying in the interior of the polydisk $\mathbb{D}^2_\nu$) can be enclosed rigorously by means of interval arithmetic, i.e.,
\begin{subequations}\label{eq:enclosure_der}
\begin{alignat}{2}
P^\star (\sigma) &\in \mathrm{Box}\left(\bar{P} (\sigma), r\right) \subset \mathbb{R}^3, \qquad
&&DP^\star(\gamma(\alpha)) \gamma'(\alpha) \in \mathrm{Box} \left(D\bar{P}(\gamma(\alpha)) \gamma'(\alpha), 2 r \beta_\nu(1)\right) \subset \mathbb{R}^3, \\
Q^\star (\theta) &\in \mathrm{Box}\left(\bar{Q} (\theta), r\right) \subset \mathbb{R}^3, \qquad
&&\partial_{\theta_j} Q^\star (\theta) \in \mathrm{Box} \left(\partial_{\theta_j} \bar{Q} (\theta), r  \beta_\nu(|\theta_j|)\right) \subset \mathbb{R}^3.
\end{alignat}
\end{subequations}
Here, $DP^\star(\gamma(\alpha)) \gamma'(\alpha) = \mathrm{i} e^{\mathrm{i}\alpha} \partial_{\sigma_1} P^\star(\sigma) - \mathrm{i} e^{-\mathrm{i}\alpha} \partial_{\sigma_2} P^\star(\sigma)$ and $|\sigma_j| = 1$, which gives the radius $2 r \beta_\nu(1)$.
The interval boxes only contain real numbers because $\bar{Q}$ has real Taylor coefficients, the stable eigenvalues being real, while $\bar{P}$ and $P^\star$ satisfy the conjugation symmetry of Remark~\ref{rem:real_image} (see the end of Section~\ref{sec:manifold}).

\paragraph{Step 2: Computing the approximate zero (with floating-point arithmetic).}
We seek an approximation $\bar{x} = (\bar{u}, \bar{\alpha}, \bar{\theta}_1, \bar{\theta}_2) \in \mathcal{X}_{C,\mu}$, with
\begin{equation}
\bar{u}^{(i)}(s) = \bar{u}^{(i)}_0 + 2 \sum_{k = 1}^K \bar{u}^{(i)}_k T_k(s), \qquad i=1,2,3.
\end{equation}
We define the truncation operator on $X_{C,\mu}$ by $(\Pi_{\le K} u)_k \bydef u_k$ for $k \le K$ and $0$ for $k > K$, and $\Pi_{>K} \bydef I - \Pi_{\le K}$; we also abbreviate $\bar{\theta} = (\bar{\theta}_1, \bar{\theta}_2)$.
Once more, we write $\bar{x} \in \Pi_{\le K} \mathcal{X}_{C,\mu}$ where $\Pi_{\le K} : \mathcal{X}_{C,\mu} \to \mathcal{X}_{C,\mu}$ is the natural extension of the truncation operator which acts component-wise, so that, for any $x = (u, \alpha, \theta_1, \theta_2)\in \mathcal{X}_{C,\mu}$ with $u=(u^{(1)}, u^{(2)}, u^{(3)}) \in X_{C,\mu}^3$,
\begin{alignat}{2}\label{eq:truncation_operators_ext}
\Pi_{\le K} x &= (\Pi_{\le K} u, \alpha, \theta_1, \theta_2), \qquad &&\Pi_{\le K} u = (\Pi_{\le K} u^{(1)}, \Pi_{\le K} u^{(2)}, \Pi_{\le K} u^{(3)}), \\
\Pi_{> K} x &= (\Pi_{> K} u, 0, 0, 0), \qquad &&\Pi_{> K} \phantom{u} = I - \Pi_{\le K}.
\end{alignat}

The solution is produced via Newton's method on the finite approximation $\Pi_{\le K} \circ F \circ \Pi_{\le K}$ of~\eqref{eq:zero-finding-heteroclinic}, where $P^\star$ and $Q^\star$ are replaced by their finite-dimensional approximations $\bar{P}$ and $\bar{Q}$.
A good initial guess is less obvious here than in the previous sections.
We obtain one by integrating the ODE numerically with the \texttt{DifferentialEquations} library \cite{DifferentialEquations.jl}, then fitting the resulting trajectory by its Chebyshev interpolation polynomial.

\paragraph{Step 3: Constructing the approximate inverse (with floating-point arithmetic).}
The Fr\'echet derivative of the map $F$ given in~\eqref{eq:zero-finding-heteroclinic} reads
\begin{equation}\label{eq:DF_het}
DF(\bar{x}) =
\begin{pmatrix}
\,\boxed{\,I - \frac{\tau}{2} \mathrm{Diag}(\mathcal{L}_C) Df(\bar{u})\,}\,
& -DP^\star(\gamma(\bar{\alpha}))\, \gamma'(\bar{\alpha})
& 0 & 0 \\[2pt]
\mathcal{E}_1 & 0 & -\partial_{\theta_1} Q^{\star,(1)}(\bar{\theta}) & -\partial_{\theta_2} Q^{\star,(1)}(\bar{\theta}) \\
\mathcal{E}_2 & 0 & -\partial_{\theta_1} Q^{\star,(2)}(\bar{\theta}) & -\partial_{\theta_2} Q^{\star,(2)}(\bar{\theta}) \\
\mathcal{E}_3 & 0 & -\partial_{\theta_1} Q^{\star,(3)}(\bar{\theta}) & -\partial_{\theta_2} Q^{\star,(3)}(\bar{\theta})
\end{pmatrix},
\end{equation}
where $\mathcal{E}_i : X_{C,\mu}^3 \to \R$ denotes the evaluation $u \mapsto u^{(i)}(1) = \mathcal{E} u^{(i)}$, with the underlying functional $\mathcal{E}$ given by the infinite row
\begin{equation}
\mathcal{E} = \begin{pmatrix} 1 & 2 & 2 & 2 & \cdots \end{pmatrix}.
\end{equation}
We argue once more that $DF(\bar{x})$ can be approximated as a finite-dimensional perturbation of the identity.
The boxed block and functionals $\mathcal{E}_j$ are the only ones acting on the infinite-dimensional space $X_{C,\mu}^3$; the remaining entries are already finite-dimensional.

The functionals $\mathcal{E}_j$ are easily bounded on the tail part.
For $\mu > 1$,
\begin{equation}\label{eq:E_tail}
\norm{\mathcal{E}\Pi_{> K}}_{\mathscr{B}(X_{C,\mu}, \R)}
= \sup_{l > K} \frac{2}{2 \mu^l}
= \frac{1}{\mu^{K+1}},
\end{equation}
which is negligible for $K$ large.

The boxed block is slightly more delicate.
Schematically, $\mathcal{L}_C$ has the matrix structure
\begin{equation*}
\mathcal{L}_C \sim
\begin{pmatrix}
\ast & \ast & \ast & \ast & \ast & \cdots \\[2pt]
\frac{1}{2} & 0 & -\frac{1}{2} & & & \\[2pt]
& \frac{1}{4} & 0 & -\frac{1}{4} & & \\[2pt]
& & \frac{1}{6} & 0 & -\frac{1}{6} & \\
& & & \ddots & \ddots & \ddots
\end{pmatrix},
\end{equation*}
where rows $k \ge 1$ are bidiagonal with $1/(2k)$ scaling and the dense first row has rapidly decaying entries.
Hence $\mathcal{L}_C$ is compact, and its tail $\Pi_{> K} \mathcal{L}_C$ and $\mathcal{L}_C \Pi_{> K}$ have small operator norm for large $K$ (an explicit bound is given in Proposition~\ref{prop:rpa_het} below).
The same logic as in Section~\ref{sec:manifold} therefore applies, and for sufficiently large $K$,
\begin{equation*}
I - \frac{\tau}{2} \mathrm{Diag}(\mathcal{L}_C) Df(\bar{u})
\approx I - \Pi_{\le K} \left( \frac{\tau}{2} \mathrm{Diag}(\mathcal{L}_C) Df(\bar{u}) \right)\Pi_{\le K}.
\end{equation*}
We set
\begin{equation}\label{eq:A_het}
A = A_{\le K} \Pi_{\le K} + \Pi_{> K}, \qquad A_{\le K} \approx \left( \Pi_{\le K} DF(\bar{x}) \Pi_{\le K} \right)^{-1}.
\end{equation}

\paragraph{Step 4: Estimating the bounds (with interval arithmetic).}

As before, the $Y$ bound consists of a finite calculation.
We conclude this section with the formula for the $Z$ bound.

\begin{pproposition}[Formula for the $Z$ bound]\label{prop:rpa_het}
Let $K \ge 1$ and $\bar{x} = (\bar{u}, \bar{\alpha}, \bar{\theta}) \in \Pi_{\le K} \mathcal{X}_{C,\mu}$.
Consider $F$ given in~\eqref{eq:zero-finding-heteroclinic} and $A$ given in~\eqref{eq:A_het}.
Consider $R \ge 0$ such that $|\bar{\theta}_j| + R < \nu$ for $j = 1, 2$, and $Z_0(R)$, $Z_1(R) \ge 0$ satisfy
\begin{subequations}
\begin{align}
\begin{aligned}
\max\Bigg(\,& \sup_{|\alpha - \bar{\alpha}| + |\theta_1 - \bar{\theta}_1| + |\theta_2 - \bar{\theta}_2| \le R}\norm{\Pi_{\le 2K+1} - \Pi_{\le 3K+2} A \Pi_{\le 3K+2} \mathcal{J}(\alpha, \theta) \Pi_{\le 2K+1}}_{\mathscr{B}(\mathcal{X}_{C,\mu})}, \\
& M \frac{|\tau|}{2} \norm{Df(\bar{u})}_{\mathscr{B}(X_{C,\mu}^3)} + \norm{A}_{\mathscr{B}(\mathcal{X}_{C,\mu})}\mu^{-(2K+2)} \Bigg)
\end{aligned} &\le Z_0(R), \\
2R \, \frac{|\tau|}{2} \norm{A}_{\mathscr{B}(\mathcal{X}_{C,\mu})} \norm{\mathcal{L}_C}_{\mathscr{B}(X_{C,\mu})} &\le Z_1(R),
\end{align}
\end{subequations}
where
\begin{align}
\mathcal{J}(\alpha, \theta) &=
\begin{pmatrix}
I - \frac{\tau}{2} \mathrm{Diag}(\mathcal{L}_C) Df(\bar{u})
& -DP^\star(\gamma(\alpha))\, \gamma'(\alpha)
& 0 & 0 \\
\mathcal{E}_1 & 0 & -\partial_{\theta_1} Q^{\star,(1)}(\theta) & -\partial_{\theta_2} Q^{\star,(1)}(\theta) \\
\mathcal{E}_2 & 0 & -\partial_{\theta_1} Q^{\star,(2)}(\theta) & -\partial_{\theta_2} Q^{\star,(2)}(\theta) \\
\mathcal{E}_3 & 0 & -\partial_{\theta_1} Q^{\star,(3)}(\theta) & -\partial_{\theta_2} Q^{\star,(3)}(\theta)
\end{pmatrix}, \\
M &= \norm{A\Pi_{\le 0}}_{\mathscr{B}(\mathcal{X}_{C,\mu})} \norm{\Pi_{\le 0}\mathcal{L}_C \Pi_{> K+1}}_{\mathscr{B}(X_{C,\mu})} + \norm{\Pi_{>K}\mathcal{L}_C \Pi_{> K+1}}_{\mathscr{B}(X_{C,\mu})},
\end{align}
and
\begin{subequations}\label{eq:bounds_het}
\begin{align}
\norm{A}_{\mathscr{B}(\mathcal{X}_{C,\mu})}
&= \max\left(\norm{A_{\le K}}_{\mathscr{B}(\mathcal{X}_{C,\mu})}, 1 \right), \\
\norm{\mathcal{L}_C}_{\mathscr{B}(X_{C,\mu})}
&= 1 + \mu, \\
\norm{\Pi_{\le 0} \mathcal{L}_C \Pi_{> K+1}}_{\mathscr{B}(X_{C,\mu})}
&= \frac{\mu^{-(K+2)}}{(K+2)^2 - 1}, \\
\norm{\Pi_{> K} \mathcal{L}_C \Pi_{> K+1}}_{\mathscr{B}(X_{C,\mu})}
&= \frac{\mu^{-1}}{2(K+1)} + \frac{\mu}{2(K+3)}.
\end{align}
\end{subequations}
Then,
\begin{equation}
\sup_{x \in B(\bar{x}, R)} \norm{I - A DF(x)}_{\mathscr{B}(\mathcal{X}_{C,\mu})} \le Z(R) = Z_0(R) + Z_1(R).
\end{equation}
\end{pproposition}

\begin{proof}
The bounds in~\eqref{eq:bounds_het} follow from the definitions of the operators.
For $x \in B(\bar{x}, R)$, the triangle inequality yields
\begin{equation*}
\norm{I - A DF(x)}_{\mathscr{B}(\mathcal{X}_{C,\mu})}
\le \norm{I - A \mathcal{J}(\alpha, \theta)}_{\mathscr{B}(\mathcal{X}_{C,\mu})} + \norm{A(DF(x) - \mathcal{J}(\alpha, \theta))}_{\mathscr{B}(\mathcal{X}_{C,\mu})},
\end{equation*}
and the second term is bounded by $Z_1(R)$ since
\begin{equation*}
DF(x) - \mathcal{J}(\alpha, \theta) =
\begin{pmatrix}
- \frac{\tau}{2} \mathrm{Diag}(\mathcal{L}_C) \big(Df(u) - Df(\bar{u})\big)
& 0 & 0 & 0 \\
0 & 0 & 0 & 0 \\
0 & 0 & 0 & 0 \\
0 & 0 & 0 & 0
\end{pmatrix}.
\end{equation*}
For the first term, since the discrete convolution product~\eqref{eq:cheb_convolution} makes $Df(\bar{u})$ a $3$-by-$3$ block operator whose entries are banded operators with bandwidth $K$, we apply the operator norm decomposition at order $2K+1$:
\begin{multline*}
\norm{I - A \mathcal{J}(\alpha, \theta)}_{\mathscr{B}(\mathcal{X}_{C,\mu})} = {} \\
\max\left( \norm{[I - A \mathcal{J}(\alpha, \theta)] \Pi_{\le 2K+1}}_{\mathscr{B}(\mathcal{X}_{C,\mu})}, \norm{[I - A \mathcal{J}(\alpha, \theta)] \Pi_{> 2K+1}}_{\mathscr{B}(\mathcal{X}_{C,\mu})} \right).
\end{multline*}
Using $A \Pi_{>2K+1} = \Pi_{>2K+1}$ together with $Df(\bar{u}) \Pi_{> 2K+1} = \Pi_{>K+1} Df(\bar{u}) \Pi_{> 2K+1}$,
\begin{equation*}
\big[ I - A \mathcal{J}(\alpha, \theta) \big] \Pi_{> 2K+1}
= \frac{\tau}{2}\, A\, \mathrm{Diag}(\mathcal{L}_C \Pi_{>K+1})\, Df(\bar{u})\, \Pi_{> 2K+1} - A \, \mathcal{E}_{\R^3} \Pi_{> 2K+1},
\end{equation*}
where $\mathcal{E}_{\R^3} \bydef (0, \mathcal{E}_1, \mathcal{E}_2, \mathcal{E}_3)$ collects the evaluation functionals.
By~\eqref{eq:E_tail}, the second term has norm at most $\norm{A}_{\mathscr{B}(\mathcal{X}_{C,\mu})} \mu^{-(2K+2)}$.
The first term has norm at most
\begin{align*}
&\frac{|\tau|}{2} \norm{A \mathcal{L}_C \Pi_{>K+1} Df(\bar{u})}_{\mathscr{B}(X_{C,\mu}^3)} \\
&\le \frac{|\tau|}{2} \norm{A \mathcal{L}_C \Pi_{>K+1}}_{\mathscr{B}(X_{C,\mu}^3)} \norm{Df(\bar{u})}_{\mathscr{B}(X_{C,\mu}^3)} \\
&= \frac{|\tau|}{2} \norm{A \Pi_{\le 0} \mathcal{L}_C \Pi_{>K+1} + A \Pi_{> 0} \mathcal{L}_C \Pi_{>K+1}}_{\mathscr{B}(X_{C,\mu}^3)} \norm{Df(\bar{u})}_{\mathscr{B}(X_{C,\mu}^3)} \\
&\le \frac{|\tau|}{2} \left(\norm{A \Pi_{\le 0} \mathcal{L}_C \Pi_{>K+1}}_{\mathscr{B}(X_{C,\mu}^3)} + \norm{\Pi_{> K} \mathcal{L}_C \Pi_{>K+1}}_{\mathscr{B}(X_{C,\mu}^3)} \right) \norm{Df(\bar{u})}_{\mathscr{B}(X_{C,\mu}^3)},
\end{align*}
by submultiplicativity, where, with a slight abuse of notation, $A$ is restricted to the $u$-components; this is legitimate since $A \Pi_{>K} = \Pi_{>K}$ and the remaining contribution is bounded by $\norm{A \Pi_{\le 0}}_{\mathscr{B}(\mathcal{X}_{C,\mu})}$.
Since $\bar{u} \in \Pi_{\le K} X_{C,\mu}$, $f$ is quadratic and $\mathcal{L}_C$ has a lower diagonal, $\mathcal{J}(\alpha, \theta) \Pi_{\le 2K+1} = \Pi_{\le 3K+2} \mathcal{J}(\alpha, \theta) \Pi_{\le 2K+1}$, so the truncated part reduces to the finite matrix $\Pi_{\le 2K+1} - \Pi_{\le 3K+2} A \Pi_{\le 3K+2} \mathcal{J}(\alpha, \theta) \Pi_{\le 2K+1}$.
\end{proof}

\begin{remark}
Note that the supremum appearing in $Z_0(R)$ only involves the terms $DP^\star(\gamma(\alpha))\, \gamma'(\alpha)$ and $DQ^\star(\theta)$.
In practice this is bounded above using the interval boxes shown in~\eqref{eq:enclosure_der}.
\end{remark}

As in Section~\ref{sec:manifold}, $A$ is a finite-rank perturbation of the identity, so $Z(R) < 1$ implies that $A$ is injective.
Carrying out the computations with the parameters given in~\cite{Code}, we verify $Z(R) < 1$ and $Y/(1-Z(R)) \le R$; the Radii Polynomial Theorem~\ref{thm:rpa} then proves the existence of the transverse heteroclinic orbit of Figure~\ref{fig:het}.


\section{Conclusion}

We have given a computer-assisted proof of a transverse heteroclinic orbit in the Shimizu--Morioka system, following the strategy of \cite{LessardMirelesJamesReinhardt2014}.
The argument splits into the validation of the equilibria and eigenpairs (Section~\ref{sec:eq_eig}), of the local invariant manifolds (Section~\ref{sec:manifold}), and of the connecting orbit (Section~\ref{sec:heteroclinic}).
Each subproblem follows the same four-step template: a zero-finding map $F$, an approximate zero $\bar{x}$, an approximate inverse $A$, and bounds $Y$ and $Z$.
The Radii Polynomial Theorem~\ref{thm:rpa}, in each case, completes the argument.
Note that the estimates favor clarity over sharpness, but tighter and more efficient bounds would matter for problems where errors compound across the three validations; we leave such improvements to the interested reader.

The analysis presented in this article can be readily adapted to other systems of autonomous ordinary differential equations.
More broadly, the maps $F$ in Sections~\ref{sec:manifold} and~\ref{sec:heteroclinic} are built around the form $x - \phi - \mathcal{L} f(x)$, with $\mathcal{L}$ compact.
This same structure appears in other zero-finding formulations of dynamical systems problems, such as initial-value problems and periodic orbits, for which analogues of Propositions~\ref{prop:rpa_manifold} and~\ref{prop:rpa_het} can be obtained by similar arguments.

On that note, the proofs of Propositions~\ref{prop:rpa_manifold} and~\ref{prop:rpa_het} exploit the quadratic nature of the Shimizu--Morioka vector field.
For higher-degree polynomial, and even non-polynomial, nonlinearities, the underlying argument still applies through the splitting
\begin{equation*}
\norm{I - A\, DF(x)}
\le \norm{I - A \mathcal{J}} + \norm{A (\mathcal{J} - DF(x))},
\end{equation*}
where $\mathcal{J} = DF(\bar{x})$ is replaced by $\mathcal{J} = I - \mathcal{L} W$, with $W \approx Df(\bar{x})$ an approximation of the multiplication operator (for instance, obtained via interpolation).
The $Z_0$ analysis is unchanged since $W$ retains a banded structure as a (possibly, block-wise) multiplication operator.
The principal new difficulty is the control of the term $\norm{W - Df(x)}$, which requires knowing how to control the nonlinearities in the relevant function space (see for instance~\cite{Lessard2016,Henot2021}, but also~\cite{Knuth1981,Jorba2005,Bucker2006,MR2531684,MR3973675}).


\section*{Acknowledgement}

O. H\'{e}not was supported by the National Science and Technology Council (NSTC) under grant No. 115-2115-M-002-001-MY2.
A. Takayasu was supported by the Japan Science and Technology Agency (JST) through the FOREST Program under grant No. JPMJFR246A, and JSPS KAKENHI under grant No. 24K00538 and No. 26K00619.

\bibliographystyle{plain}
\bibliography{references}

\end{document}